\documentclass[11pt]{amsart}
\usepackage{amsmath, amssymb, amscd, mathrsfs, url, pinlabel,hyperref, verbatim}
\usepackage[margin=1.25in,centering,letterpaper,dvips]{geometry}
\usepackage{color,dcpic,latexsym,graphicx,epstopdf,comment}
\usepackage[all]{xy}
\usepackage{enumitem}

\newtheorem {theorem}{Theorem}
\newtheorem {lemma}[theorem]{Lemma}
\newtheorem {proposition}[theorem]{Proposition}
\newtheorem {corollary}[theorem]{Corollary}

\newtheorem {definition}[theorem]{Definition}

\theoremstyle{remark}

\newtheorem {remark}[theorem]{Remark}

\numberwithin{equation}{section}
\numberwithin{theorem}{section}

\newlist{pcases}{enumerate}{1}
\setlist[pcases]{
  label=\bf{Case~\arabic*:}\protect\thiscase.~,
  ref=\arabic*,
  align=left,
  labelsep=0pt,
  leftmargin=0pt,
  labelwidth=0pt,
  parsep=0pt
}
\newcommand{\case}[1][]{%
  \if\relax\detokenize{#1}\relax
    \def\thiscase{}%
  \else
    \def\thiscase{~#1}%
  \fi
  \item
}

\newcommand{\ZZ}{\mathbb{Z}}
\newcommand{\Z}{\mathbb{Z}}

\newcommand{\F}{\mathbb{F}}
\newcommand{\Q}{\mathbb{Q}}

\newcommand{\spc}{\operatorname{Spin}^c}
\newcommand{\spinc}{\mathfrak{s}}
\newcommand{\spint}{\mathfrak{t}}
\newcommand{\tower}{\mathcal{T}^+}

\newcommand{\rpthree}{\mathbb{RP}^3}
\newcommand{\dcover}{\Sigma}

\newcommand{\img}{\operatorname{Im}}

\newcommand{\hfp}{HF^+}
\newcommand{\hfred}{HF^+_{\mathrm{red}}}
\newcommand{\hfinfty}{HF^\infty}
\newcommand{\hfcirc}{HF^\circ}

\newcommand{\gr}{\operatorname{gr}}

\specialcomment{steven}{%
  \begingroup\color{blue}  Steven \; $\blacktriangleright$ \;}{%
  \endgroup}
\specialcomment{tye}{%
  \begingroup\color{red} Tye \; $\blacktriangleright$ \;}{%
  \endgroup}

\title{Quasi-alternating links with small determinant}
\date{}

\begin{document}

\author{Tye Lidman}
\email{tlid@math.utexas.edu}
\address{Mathematics Department\\The University of Texas at Austin\\RLM 8.100\\2515 Speedway Stop C1200\\Austin, TX 78712-1202}

\author{Steven Sivek}
\email{ssivek@math.princeton.edu}
\address{Department of Mathematics\\Princeton University\\Fine Hall, Washington Road\\Princeton, NJ 08544-1000}

\begin{abstract}
Quasi-alternating links of determinant 1, 2, 3, and 5 were previously classified by Greene and Teragaito, who showed that the only such links are two-bridge.  In this paper, we extend this result by showing that all quasi-alternating links of determinant at most 7 are connected sums of two-bridge links, which is optimal since there are quasi-alternating links not of this form for all larger determinants.  We achieve this by studying their branched double covers and characterizing distance-one surgeries between lens spaces of small order, leading to a classification of formal L-spaces with order at most 7.
\end{abstract}

\maketitle


\section{Introduction}\label{sec:intro}

Quasi-alternating links are a natural generalization of non-split alternating links which have received a considerable amount of attention over the past decade.  First introduced by Ozsv\'ath-Szab\'o in \cite{OzsvathBranched}, these links provide a more general family of links whose Khovanov homology and knot Floer homology are particularly simple -- they are {\em homologically thin} \cite{ManolescuOzsvath} -- and they have exhibited a number of other behaviors found in alternating links.  This can also be translated into topological applications: for example, any branched double cover of a quasi-alternating link is an example of a manifold which cannot admit a co-orientable taut foliation \cite{OzsvathBranched}.  

For many invariants such as the Alexander polynomial, there are only finitely many alternating knots which attain a fixed value, and this has had applications to Dehn surgery questions (most recently, \cite{Gainullin,LackenbyPurcell}).  For quasi-alternating links, it is harder to obtain such finiteness results.  For instance, it is still unknown if there are only finitely many quasi-alternating links of any fixed determinant.

In \cite{GreeneNonQA}, Greene classifies quasi-alternating links with determinant one, two, and three as the unknot, Hopf link, and the two trefoils respectively.  Teragaito \cite[Theorem~1.9]{Teragaito-K} completes the classification of quasi-alternating links with determinant 5 by a different method -- they are the torus knots $T_{2,\pm 5}$ and the figure eight -- and points out that the classification of quasi-alternating links of determinant 4 is still unknown despite partial results in this case \cite{Teragaito-Q}, but he conjectures that the only such links should be the torus links $T_{2,\pm4}$ and the connected sum of two Hopf links.  The goal of this paper is to verify this and to also classify quasi-alternating links with determinants up to 7 as well.  Specifically, we prove:

\begin{theorem}\label{thm:classification}
Let $L$ be a quasi-alternating link with determinant at most 7.  Then $L$ is either two-bridge or a connected sum of two-bridge links.
\end{theorem}
In particular, $L$ is either the unknot, the figure-eight knot, the torus link $T_{2,n}$ with $2 \leq |n| \leq 7$, a connected sum of two Hopf links, a connected sum of a trefoil with a Hopf link, or either the $5_2$ knot or its mirror.  This answers a question of Teragaito \cite[Conjecture~1.10]{Teragaito-K}:
\begin{corollary}
If $L$ is a quasi-alternating link which is not alternating, then $\det(L) \geq 8$.
\end{corollary}

Theorem~\ref{thm:classification} also implies the following about positive knots.
\begin{corollary}
Let $K$ be a nontrivial positive knot of genus at most 2 and determinant at most 7.  Then $K$ is either $T_{2,3}$, $T_{2,5}$, or $5_2$.
\end{corollary}

\begin{proof}
Positive knots of genus at most 2 are quasi-alternating \cite{JongKishimoto}.  Now apply Theorem~\ref{thm:classification} to conclude that $K$ must be one of the knots listed above.
\end{proof}

The main idea of the proof is that used originally by Greene, and by Teragaito \cite{Teragaito-Q} in the determinant 4 case, which is to lift to the branched double cover and rephrase the problem in terms of Dehn surgery.  Greene and Levine \cite{GreeneLevine} define a notion of \emph{formal L-space} (see Definition~\ref{def:formal-l-space}) which is meant to be a 3-manifold analogue of quasi-alternating links; and indeed the branched double cover of any quasi-alternating link is a formal L-space.  Theorem~\ref{thm:classification} will be a consequence of the following classification result.

\begin{theorem}\label{thm:formal-l-spaces}
If $Y$ is a formal L-space with $|H_1(Y;\ZZ)| \leq 7$, then $Y$ is a connected sum of lens spaces.
\end{theorem}

In fact, Teragaito nearly completes the determinant-4 classification of Theorem~\ref{thm:classification} in \cite[Lemma 2.3]{Teragaito-Q}; he could obtain the desired conclusion if he knew that non-trivial knots in $\rpthree$ cannot have non-trivial distance-one $\rpthree$ surgeries.  The key technical result which allows us to extend Teragaito's work is a recent theorem of Gainullin \cite{Gainullin-complement} giving a Dehn surgery characterization of the unknot for nullhomologous knots in L-spaces, which extends a result of Kronheimer-Mrowka-Ozsv\'ath-Szab\'o for knots in $S^3$ \cite{KMOS}.  We remark that while the case of $|H_1(Y;\mathbb{Z})| \leq 3$ in Theorem~\ref{thm:formal-l-spaces} follows from Greene's work, Teragaito's classification of determinant 5 quasi-alternating links does not determine the order 5 formal L-spaces.

One could likely use the arguments in this paper to extend the classification of quasi-alternating links to slightly larger determinants.  However, the obstruction to continuing this process for all values of the determinant is that one needs the classification of lens space surgeries on knots (the {\em Berge conjecture}) as well as a complete list of links whose branched double cover gives a fixed manifold.       

Moreover, for any $n \geq 8$ there are quasi-alternating links of determinant $n$ which are not connected sums of two-bridge links.  Indeed, if $k\geq 2$ then the $(2,k,-3)$-pretzel link is quasi-alternating \cite[Theorem~3.2]{CK} but not alternating.  One can show that its branched double cover is $(k+6)$-surgery on the right handed trefoil, which implies that $\det(P(2,k,-3)) = k+6$.  Note that this branched double cover is a formal L-space.  This is because $S^3_7(T_{2,3})$ is the lens space $L(7,4)$, the branched double cover of a two-bridge link, hence quasi-alternating, and $S^3_{p+1}(K)$ is a formal L-space whenever $S^3_p(K)$ is.

\begin{remark}
Here, and throughout the rest of this paper, we use the convention that $L(p,q)$ is $\frac{p}{q}$-surgery on the unknot in $S^3$.  We will also write $\F = \Z/2\Z$.
\end{remark}

\subsection*{Outline} In Section~\ref{sec:background} we recall the definition and properties of quasi-alternating links, show that Theorem~\ref{thm:classification} follows from Theorem~\ref{thm:formal-l-spaces}, and review other related material necessary for the setup.  In Sections~\ref{sec:consecutive} and \ref{sec:surgery25} we prove some results characterizing when knot complements in lens spaces of small order can have distance-one fillings which are also lens spaces of small order.  Finally, we use these results in Section~\ref{sec:classification} to prove Theorem~\ref{thm:formal-l-spaces}.

\subsection*{Acknowledgments} We would like to thank Fyodor Gainullin for helpful discussions.  The first author was partially supported by NSF RTG grant DMS-1148490.  The second author was supported by NSF grants DMS-1204387 and DMS-1506157.

\section{Quasi-alternating links, branched double covers, and surgery}\label{sec:background}

We begin by recalling the definitions of quasi-alternating links and formal L-spaces.
\begin{definition}[\cite{OzsvathBranched}] \label{def:qa}
The set $\mathcal{Q}$ of \emph{quasi-alternating links} is the smallest set of links in $S^3$ containing the unknot such that for any link $L$, if $L$ admits a diagram with a crossing whose two resolutions $L_0,L_1$ satisfy
\begin{itemize}
\item $L_0,L_1 \in \mathcal{Q}$,
\item $\det(L) = \det(L_0) + \det(L_1)$,
\end{itemize}
then $L \in \mathcal{Q}$.
\end{definition}
In particular, all non-split alternating links are quasi-alternating  \cite{OzsvathBranched}, and $\mathcal{Q}$ is also closed under taking mirrors and connected sums.  (This last claim follows for $K \# L$ by induction on $\det(L)$: if $K \in \mathcal{Q}$ then $K \# U \in \mathcal{Q}$, and given resolutions $L_0,L_1$ of $L$ as above we have $K\#L_0, K\#L_1 \in \mathcal{Q}$ by hypothesis, so $K\#L\in\mathcal{Q}$ as well.)

We say a collection of closed, oriented 3-manifolds $(Y_1,Y_2,Y_3)$ forms a \emph{triad} if there is a 3-manifold $M$ with torus boundary and a collection of oriented curves $\gamma_1,\gamma_2,\gamma_3 \subset \partial M$ at pairwise distance 1 such that each $Y_i$ is the result of Dehn filling along $\gamma_i$.  This is precisely the condition under which the Heegaard Floer homologies of the $Y_i$ (in some order) fit into a surgery exact triangle.  We will define $\det(Y)$ to be $|H_1(Y;\ZZ)|$ if $b_1(Y) = 0$ and $0$ otherwise; note that if $L$ is a link then its branched double cover satisfies $\det(\dcover(L)) = \det(L)$.

\begin{definition}[{\cite[Section~7]{GreeneLevine}}] \label{def:formal-l-space}
The set $\mathcal{F}$ of \emph{formal L-spaces} is the smallest set of rational homology 3-spheres containing $S^3$ such that whenever $(Y,Y_0,Y_1)$ is a triad with $Y_0,Y_1\in\mathcal{F}$ and
\[ \det(Y) = \det(Y_0) + \det(Y_1), \]
we have $Y \in \mathcal{F}$ as well.
\end{definition}

This definition can be interpreted as a 3-manifold analogue of the notion of a quasi-alternating link.  Indeed, given any triple of links $(L,L_0,L_1)$ as in Definition~\ref{def:qa}, the branched double covers $(\dcover(L),\dcover(L_0),\dcover(L_1))$ form a triad, and $\dcover(U) = S^3$, so the branched double cover of any quasi-alternating link is a formal L-space.  It is easy to see that $\mathcal{F}$ contains all lens spaces and is closed under orientation reversal and taking connected sums.

We recall that we are interested in classifying quasi-alternating links and formal L-spaces with small determinant.  This classification has been carried out for determinant at most 3 by work of Greene.  

\begin{theorem}[\cite{GreeneNonQA}]\label{thm:greene-classification}
If $Y$ is a formal L-space with $\det(Y)$ equal to 1, 2, or 3, then $Y$ is $S^3$, $\rpthree$, or $\pm L(3,1)$ respectively.
\end{theorem}

In order to deduce the classification of quasi-alternating links from the classification of formal L-spaces, we appeal to the following.

\begin{theorem}[Hodgson-Rubinstein \cite{HR}]\label{thm:hr}
If $L\subset S^3$ is a link whose branched double cover is the lens space $L(p,q)$, then $L$ is the two-bridge link with continued fraction equal to $p/q$.
\end{theorem}

\begin{theorem}[Kim-Tollefson \cite{KimTollefson}]\label{thm:kt-prime}
If the branched double cover of a non-split link $L \subset S^3$ is a connected sum $Y_1 \# Y_2$, with each $Y_i$ prime, then $L$ is a connected sum $L_1 \# L_2$ of links such that $\dcover(L_1)=Y_1$ and $\dcover(L_2)=Y_2$.
\end{theorem}

We now show that Theorem~\ref{thm:classification} follows from Theorem~\ref{thm:formal-l-spaces}.  
\begin{proof}[Proof of Theorem~\ref{thm:classification}]
Recall that if $L$ is a quasi-alternating link, then $\dcover(L)$ is a formal L-space; if $\det(L) \leq 7$, then $\dcover(L)$ is a connected sum of lens spaces by Theorem~\ref{thm:formal-l-spaces}.  It follows that $L$ is a connected sum of two-bridge links, and these links are determined uniquely by the lens space summands, so Theorem~\ref{thm:classification} follows immediately from Theorem~\ref{thm:formal-l-spaces}.
\end{proof}

Thus, the remainder of the paper is devoted to proving Theorem~\ref{thm:formal-l-spaces}.  

\subsection{Some general facts about surgery}
The following lemma will be useful in the proof of Theorem~\ref{thm:formal-l-spaces}.  We first recall that a knot $K$ in a rational homology sphere $Y$ is {\em primitive} if it generates $H_1(Y)$.  If $M$ denotes the exterior of $K$, then primitivity implies that $H_1(M) \cong \mathbb{Z}$ and there exists a curve $\mu$ on $\partial M$ which represents the generator of $H_1(M)$.  Further, we have that $\mu$ and the rational longitude $\lambda$ form a basis for $H_1(\partial M)$, and so $\Delta(\mu,\lambda) = 1$.  (Recall that the rational longitude is the unique slope on the boundary of a rational homology solid torus $P$ which is torsion in $H_1(P)$.) Given an arbitrary slope $\alpha$, we have $|H_1(M(\alpha))| = \Delta(\alpha,\lambda)$.  

\begin{lemma}\label{lem:distance-one}
Let $K$ be a primitive knot in a rational homology sphere $Y$.  If a non-trivial filling on the exterior $M$ results in a rational homology sphere $Y'$, possibly homeomorphic to $Y$, then the distance from this filling slope to the trivial slope is a multiple of $\gcd(|H_1(Y)|, |H_1(Y')|)$.  In particular, if $|H_1(Y)|$ and $|H_1(Y')|$ are not relatively prime, then such a filling cannot have distance one from the trivial slope.
\end{lemma}
\begin{proof}
Let $\gamma$ and $\eta$ be slopes on $M$ for which Dehn filling yields $Y$ and $Y'$ respectively, and let $p=|H_1(Y)|$ and $q=|H_1(Y')|$.  Then we can write $\gamma = p\mu + a\lambda$ and $\eta = q\mu + b\lambda$ for some integers $a$ and $b$, since $\Delta(\gamma,\lambda) = p$ and $\Delta(\eta,\lambda) = q$.  Then $\Delta(\gamma,\eta) = |pb-qa|$, which is clearly a multiple of $\gcd(p,q)$ as claimed.
\end{proof}

\begin{remark}\label{rmk:core-distance}
In particular, if $K$ represents a core of a genus one Heegaard splitting of $L(p,q)$, then $K$ is primitive.  It follows from Lemma~\ref{lem:distance-one} that no fillings which are distance one from the trivial filling can yield $L(p,q')$ for any $q'$.
\end{remark}

We will occasionally make use of the Casson-Walker invariant \cite{BoyerLines,Walker} in order to study manifolds arising from several surgeries on the same knot.  This invariant $\lambda(Y) \in \Q$ agrees with the usual $\ZZ$-valued Casson invariant if $Y$ is a homology sphere, and it satisfies a surgery formula for a knot $K$ in a homology sphere $Y$:
\begin{equation}\label{eq:cw-surgery-alternate}
\lambda(Y_{a/b}(K)) - \lambda(S^3_{a/b}(U)) = \lambda(Y) + \frac{b}{a}A(K).
\end{equation}
Here $A(K) = \frac{\Delta''_K(1)}{2}$, where we normalize the Alexander polynomial so that $\Delta_K(t^{-1})=\Delta_K(t)$ and $\Delta_K(1) = 1$.

The Casson-Walker invariant $\lambda(Y)$ is related to the Heegaard Floer $d$-invariants \cite{OzsvathGraded} by the following formula, which was first proved by Ozsv\'{a}th--Szab\'{o} \cite[Theorem~1.3]{OzsvathGraded} for homology spheres and then generalized to all rational homology spheres by Rustamov \cite[Theorem~3.3]{Rustamov}.

\begin{theorem}[{\cite{OzsvathGraded,Rustamov}}] \label{thm:cw-sum-d}
If $Y$ is a rational homology sphere, then 
\[ |H_1(Y;\ZZ)| \cdot \lambda(Y) = \sum_{\spint \in \spc(Y)} \left(\chi(\hfred(Y,\spint)) - \frac{1}{2}d(Y,\spint)\right). \]
\end{theorem}

Another useful tool for us will be the linking form.  For notation, if $H_1(Y)$ is cyclic of order $p$, then we will say that $Y$ has linking form $\frac{a}{p}$ if there exists a generator with self-linking $\frac{a}{p}$.  We note for cyclic groups, the two forms $\frac{a}{p}$ and $\frac{b}{p}$ are equivalent if and only if $a \equiv u^2 b \pmod{p}$ for a unit $u \in \Z/p\Z$.  In this notation, if $K$ is a knot in a homology sphere $Y$, then the linking form of $Y_{p/q}(K)$ is $\frac{q}{p}$.   

\subsection{Heegaard Floer homology and surgeries}

If $\spint$ is a torsion $\spc$ structure on $Y$, the Heegaard Floer invariants $\hfcirc(Y,\spint)$ ($\circ = +,-,\infty$) admit an absolute $\Q$-grading \cite{OzsvathSmooth}.  This has already appeared implicitly in the statement of Theorem~\ref{thm:cw-sum-d}: given a rational homology sphere $Y$, the $d$-invariants $d(Y,\spint) \in \Q$ are defined in \cite{OzsvathGraded} as the lowest grading of a nonzero element $x\in \hfp(Y,\spint)$ such that $x \in \img(U^k)$ for all $k \geq 0$.  In this subsection we will review some properties of surgeries and their relationship to gradings in Heegaard Floer homology.

\begin{theorem}[Ozsv\'ath-Szab\'o \cite{OzsvathSmooth}]\label{thm:os-gradings}
The absolute grading on $\hfcirc(Y,\spint)$ has the following properties:
\begin{itemize}
\item If $(W,\spinc)$ is a $\spc$ cobordism from $(Y_1,\spint_1)$ to $(Y_2,\spint_2)$, with $\spint_1,\spint_2$ torsion, then the induced map $F^\circ_{W,\spinc}: \hfcirc(Y_1,\spint_1) \to \hfcirc(Y_2,\spint_2)$ changes grading by 
\[ \gr(F^\circ_{W,\spinc}) = \frac{c_1(\spinc)^2 - 2\chi(W) - 3\sigma(W)}{4}. \]
\item The natural map $\pi: \hfinfty(Y,\spint) \to \hfp(Y,\spint)$ preserves the absolute grading.
\item The $U$-action on $HF^\circ$ has degree $-2$.
\end{itemize}
\end{theorem}

The cobordism maps on $\hfinfty$ are particularly simple.  If $b_2^+(W) > 0$, then the map $F^\infty_{W,\spinc}$ is zero for all $\spinc$ \cite[Lemma~8.2]{OzsvathSmooth}.  On the other hand, if $W$ is a 2-handle cobordism with $b_2^-(W)=1$ between rational homology spheres $Y_1$ and $Y_2$, and $\spinc \in \spc(W)$ restricts to $Y_i$ as $\spint_i$ for $i=1,2$, then
\[ F^\infty_{W,\spinc}: \hfinfty(Y_1,\spint_1) \to \hfinfty(Y_2,\spint_2) \]
is an isomorphism \cite[Proposition~9.4]{OzsvathGraded} between $\F[U]$-modules of the form $\F[U,U^{-1}]$, so it is determined completely by its grading, which in this case simplifies to $\frac{1}{4}(c_1(\spinc)^2 + 1)$.  Moreover, in the case $b_2^-(W)=1$, the elements of $\hfp(Y_i,\spint_i)$ which determine $d(Y_1,\spint_1)$ and $d(Y_2,\spint_2)$ are both in $\pi(\hfinfty(Y_i,\spint_i)) \subset \hfp(Y_i,\spint_i)$ by definition, so we must have
\begin{equation} \label{eq:surgery-d-relation}
d(Y_2,\spint_2) - d(Y_1,\spint_1) \equiv \gr(F^\infty_{W,\spinc}) = \frac{c_1(\spinc)^2+1}{4} \pmod{2}.
\end{equation}

\begin{lemma} \label{lem:surgery-fplus}
Let $W$ be a 2-handle cobordism between rational homology spheres $Y_1$ and $Y_2$.  If $b_2^+(W) = 1$ and $Y_1$ is an L-space, then the map $F^+_{W,\spinc}: \hfp(Y_1,\spinc|_{Y_1}) \to \hfp(Y_2,\spinc|_{Y_2})$ is zero.  If instead $b_2^-(W) = 1$ and $Y_2$ is an L-space, then $F^+_{W,\spinc}$ is surjective.
\end{lemma}

\begin{proof}
Write $\spint_i = \spinc|_{Y_i}$, $i=1,2$.  The cobordism maps $F^\circ_{W,\spinc}$ fit into a commutative diagram 
\[ \xymatrix{
\hfinfty(Y_1,\spint_1) \ar[r]^{F^\infty_{W,\spinc}} \ar[d]_{\pi_1} & \hfinfty(Y_2,\spint_2) \ar[d]^{\pi_2} \\
\hfp(Y_1,\spint_1) \ar[r]^{F^+_{W,\spinc}} & \hfp(Y_2,\spint_2).
} \]
If $b_2^+(W)=1$ and $Y_1$ is an L-space, then $F^\infty_{W,\spinc} = 0$ and so $F^+_{W,\spinc}\circ\pi_1 = 0$, but $\pi_1$ is surjective so we must have $F^+_{W,\spinc}=0$.  On the other hand, if $b_2^-(W)=1$ and $Y_2$ is an L-space, then $F^\infty_{W,\spinc}$ is an isomorphism and $\pi_2$ is surjective, so their composition $\pi_2 \circ F^\infty_{W,\spinc} = F^+_{W,\spinc} \circ \pi_1$ is surjective, hence $F^+_{W,\spinc}$ is surjective as well.
\end{proof}

The cobordism maps induced by various surgeries on a knot (namely, any three which form a triad) fit into exact triangles; the surgery exact triangle first appeared in \cite{Ozsvath2004a}, but we cite the version from \cite{OzsvathRational}.

\begin{theorem}[Ozsv\'ath-Szab\'o {\cite[Theorem~6.2]{OzsvathRational}}] \label{thm:surgery-triangle}
Let $K \subset Y$ be a rationally null-homologous knot with framing $\lambda$ and meridian $\mu$.  Then there is a long exact sequence
\[ \dots \to \hfp(Y_\lambda(K)) \to \hfp(Y_{\lambda + \mu}(K)) \to \hfp(Y) \to \dots, \]
in which the maps are the $\hfp$ cobordism maps corresponding to attaching 2-handles.
\end{theorem}

The maps in the exact triangle are described via holomorphic triangle counts in \cite{Ozsvath2004a,OzsvathRational}, but the claim that they are cobordism maps follows from the fact that 2-handle cobordism maps are defined in \cite{OzsvathSmooth} using precisely these counts.  Each cobordism comes from attaching a 2-handle to a meridian of the core of the previous surgery.

The signature of each 2-handle cobordism in the exact triangle can be computed according to \cite[Lemma~42.3.1]{KMbook}: let $Z \subset Y$ be the exterior of $K$, and $\lambda \subset \partial Z$ a primitive oriented curve which generates $\ker(H_1(\partial Z) \to H_1(Z))$.  If $W$ is the 2-handle cobordism from the Dehn filling of $Z$ along $\gamma$ to the Dehn filling along $\gamma'$, where $\gamma$ and $\gamma'$ are oriented so that $\gamma \cdot \gamma' = -1$, then $W$ has signature $+1$ (respectively $-1$) if $\gamma\cdot \lambda$ and $\gamma'\cdot\lambda$ have the same sign (respectively opposite signs).  (If either of these is zero then $\sigma(W)=0$.)  If all three manifolds in the triangle are rational homology spheres, it follows that two of the cobordisms are negative definite and one is positive definite.

For example, if we let $Y' = Y_2(K)$ and let $K' \subset Y'$ be the image of an $n$-framed meridian of $K$, then the surgery triangle corresponding to $(Y',K')$ has the form
\begin{equation}\label{eq:2-exact-triangle}
\dots \to \hfp(Y_{2-\frac{1}{n}}(K)) \to \hfp(Y_{2-\frac{1}{n+1}}(K)) \to \hfp(Y_{2}(K)) \to \dots.
\end{equation}
We can verify that the cobordism from $Y_2(K)$ to $Y_{2-\frac{1}{n}}(K)$ is positive definite for $n\geq 1$.  Indeed, if $\mu$ and $\lambda$ are the meridian and longitude of $K$ in its exterior $Z$, then the boundary orientation of $\partial Z$ gives $\mu\cdot\lambda=-1$.  The manifolds $Y_2(K)$ and $Y_{2-\frac{1}{n}}(K)$ are Dehn fillings along $\gamma = 2\mu+\lambda$ and $\gamma'=(2n-1)\mu+n\lambda$ respectively, satisfying $\gamma\cdot\gamma'=-1$.  We have $\gamma\cdot\lambda = -2$ and $\gamma'\cdot\lambda=-(2n-1)$, and since $n\geq 1$ these have the same sign.

Finally, we use Dehn surgery to verify the following property of the $d$-invariants.

\begin{lemma} \label{lem:d-2p-Z}
If $H_1(Y;\ZZ) \cong \ZZ/p\ZZ$, then $d(Y,\spint) \in \frac{1}{2p}\ZZ$ for all $\spint \in \spc(Y)$.  Further, if the linking form of $Y$ is equivalent to that of $L(p,q)$, then the set of $d$-invariants of $Y$ agrees mod 2 with that of $L(p,q)$.  
\end{lemma}
\begin{proof}
We first claim that $d(L(p,q), \spint) \in \frac{1}{2p}\ZZ$ for all $\spint \in \spc(Y)$.  By \cite[Theorem 4]{Tange}, we can write 
\[
d(L(p,q),i) = 3s(q,p) + \frac{n}{2p}
\]
for some $n \in \mathbb{Z}$, where $s(q,p)$ is a Dedekind sum.  Since $s(q,p) \in \frac{1}{6p}\ZZ$ when $\gcd(p,q) = 1$ \cite{RademacherGrosswald}, we have the desired claim.  

Now, fix $Y$ with $H_1(Y;\ZZ) \cong \ZZ/p\ZZ$ and suppose the linking form of $Y$ is equivalent to $\frac{q}{p}$.  Then it follows from \cite[Corollary 3.9 and Proposition 3.17]{CochranGergesOrr} that we can obtain $Y$ from $L(p,q)$ by a sequence of $\pm 1$-surgeries on nullhomologous knots.  Therefore, it suffices to show that $\pm 1$-surgery on a nullhomologous knot in a rational homology sphere preserves the set of $d$-invariants mod 2.  Here we will implicitly use the canonical identification between $\spc(Y)$ and $\spc(Y_{\pm 1}(K))$.

If $Y_2$ is obtained by $-1$-surgery on $Y_1$, then we consider the corresponding 2-handle cobordism $W$ from $Y_1$ to $Y_2$ which is negative definite.  By \eqref{eq:surgery-d-relation}, it suffices to show that for each $\spinc \in \spc(W)$, the grading shift $\gr(F^\infty_{W,\spinc})$ is an even integer.  We have
\[
\gr(F^\infty_{W,\spinc}) = \frac{c_1(\spinc)^2 + 1}{4} = \frac{-(2k-1)^2 + 1}{4}
\]          
for some integer $k$, and since $(2k-1)^2 \equiv 1 \pmod{8}$, it follows that $d(Y_1, \spint_1) \equiv d(Y_2, \spint_2) \pmod{2}$, where $\spint_i = \spinc|_{Y_i}$.  On the other hand, if $Y_2$ is obtained from $Y_1$ by $+1$-surgery on a knot, then $Y_1$ is obtained from $Y_2$ by $-1$-surgery on the dual knot, and we can apply the same argument.  Since we have shown that the $d$-invariants of $L(p,q)$ are in $\frac{1}{2p}\ZZ$, we must have the same for $Y$.  
\end{proof}

\subsection{Lens space surgeries and L-spaces}

Greene states Theorem~\ref{thm:greene-classification} in \cite{GreeneNonQA} as a classification of quasi-alternating links of determinant up to 3, namely that they are the unknot, the Hopf link, or a trefoil; but his proof, which passes through branched double covers, actually establishes Theorem~\ref{thm:greene-classification} as stated here.  Theorem~\ref{thm:hr} then yields the classification for the branch sets.  We recall the argument here in order to suggest how our own classification will proceed and to introduce some necessary results.

The only formal L-space with determinant 1 is $S^3$ by definition.  To classify formal L-spaces $Y$ with $\det(Y)=2$, Greene observes that there must be a triad $(Y,Y_0,Y_1)$ with $\det(Y_0) = \det(Y_1) = 1$, and hence $Y_0=Y_1=S^3$; thus both $Y$ and $Y_1=S^3$ result from nontrivial surgeries on a knot $J \subset Y_0=S^3$, and only the unknot in $S^3$ has a nontrivial $S^3$ surgery.  Therefore, $J$ is the unknot and $Y$ is therefore $\rpthree$.  In order to classify $Y$ with $\det(Y)=3$, we must likewise have a triad $(Y, S^3, \rpthree)$, so both $Y$ and $\rpthree$ arise as surgeries on the same knot $J \subset S^3$; but the only knot in $S^3$ with an $\rpthree$ surgery is the unknot, so $Y$ must be a lens space of order three.

The key input needed for the above argument is an understanding of which knots in $S^3$ have lens space surgeries.  While this is a difficult question in general, it is understood for lens spaces of small order.

\begin{theorem}[{\cite[Corollary~8.4]{KMOS}}]\label{thm:kmos}
Suppose that $S^3_{p/q}(K)$ results in a lens space for $|p| \leq 8$.  Then $K$ is the unknot or a trefoil.  In particular, if $|p| \leq 4$ then $K$ is the unknot.
\end{theorem} 

Similarly, we will occasionally need to understand when knots in Heegaard Floer L-spaces have lens space surgeries.  We will use a recent result of Gainullin \cite{Gainullin-complement}, who proved the following Dehn surgery characterization of the unknot, generalizing the main result of \cite{KMOS} (see also \cite[Corollary 1.3]{OzsvathRational}).

\begin{theorem}[{\cite[Theorem~8.2]{Gainullin-complement}}] \label{thm:surgery-characterization}
Let $K$ be a nullhomologous knot in an L-space $Y$.  If $HF^+(Y_{p/q}(K))$ and $HF^+(Y_{p/q}(U))$ are isomorphic as absolutely-graded $\F[U]$-modules, then $K$ is the unknot.
\end{theorem}  

It follows that nullhomologous knots in L-spaces are determined by their complements.  In our applications of Theorem~\ref{thm:surgery-characterization}, we will have the stronger assumption that $Y_{p/q}(K)$ and $Y_{p/q}(U)$ are orientation-preserving homeomorphic, except in the case when $Y$ is a homology sphere L-space.  We remark that in the case of a homology sphere, the result can be proved exactly as in \cite[Corollary 1.3]{OzsvathRational}.  (Here, the only change necessary from the case $Y=S^3$ is to use \cite{NiNote}, which shows that knot Floer homology detects the unknot in homology spheres.)  

Under some additional mild conditions, Theorem~\ref{thm:surgery-characterization} yields two stronger results.  

\begin{corollary}\label{cor:lens-cosmetic}
Let $K$ be a knot in an L-space $Y$ with $|H_1(Y;\Z)|$ prime.  If a non-trivial surgery on $K$ which is distance one from the trivial surgery gives $Y$, then $K$ is the unknot.   
\end{corollary}
\begin{proof}
If $K$ is nullhomologous, then the result follows from Theorem~\ref{thm:surgery-characterization}.  On the other hand, if $K$ is not nullhomologous, we see that $K$ is primitive, because $|H_1(Y)|$ is prime.  Therefore, it follows from Lemma~\ref{lem:distance-one} that the distance between the trivial and non-trivial surgeries is a multiple of $|H_1(Y)|$, which is a contradiction.      
\end{proof}

\begin{corollary} \label{cor:invisible-lens}
Let $Y$ be an {\em invisible} 3-manifold, i.e.\ a homology sphere L-space with $d(Y)=0$.  If there is a knot $K\subset Y$ and a rational number $\frac{p}{q}$ satisfying $Y_{p/q}(K) = L(p,q)$, then $Y = S^3$ and $K$ is the unknot.
\end{corollary}

\begin{proof}
We have $\hfp(Y_{p/q}(K)) = \hfp(L(p,q)) = \hfp(Y \# L(p,q)) = \hfp(Y_{p/q}(U))$ by the K\"unneth formula \cite{Ozsvath2004a} for $\hfp$ and the fact that $\hfp(Y) = \hfp(S^3)$.  Now $K$ is the unknot by Theorem~\ref{thm:surgery-characterization}, so $L(p,q) = Y_{p/q}(K) = Y\#L(p,q)$ implies that $Y=S^3$.
\end{proof}


\section{Primitive knot complements with several lens space fillings}
\label{sec:consecutive}

In this section we will prove the following theorem, which we will use to study when two lens spaces $L(p,1)$ and $L(p+1,1)$ can belong to a triad.

\begin{theorem}
\label{thm:lp1-fillings}
Let $M$ be a homology $S^1 \times D^2$ with torus boundary, and suppose there are a pair of slopes $\gamma,\gamma'\subset \partial M$ with $\Delta(\gamma,\gamma')=1$ and a positive integer $p\not\equiv 1\pmod{12}$ such that Dehn filling along $\gamma$ and $\gamma'$ produces the lens spaces $L(p,1)$ and $L(p+1,1)$ respectively.  Then $M$ is a solid torus.
\end{theorem}

Let $\lambda \subset \partial M$ be a closed, oriented curve such that $[\lambda]$ generates $\ker(H_1(\partial M) \to H_1(M))$, and declare a meridian to be any oriented curve $\mu \subset \partial M$ such that $\mu \cdot \lambda = 1$.  For convenience, given a simple, closed curve $s \subset \partial M$ we will also use $s$ to denote its homology class in $H_1(\partial M) \cong \ZZ^2$ and its induced slope.  More generally, the notation $\lambda$ for a slope will now always be used to refer to the rational longitude.  

\begin{lemma} \label{lem:lp1f-meridian}
Given $M,\gamma,\gamma'$ satisfying the hypotheses of Theorem~\ref{thm:lp1-fillings}, there is a meridian $\mu$ and a sign $\epsilon \in \{\pm 1\}$ such that $\gamma = p\mu + \epsilon\lambda$ and $\gamma' = (p+1)\mu + \epsilon\lambda$.
\end{lemma}

\begin{proof}
Fix an initial choice $\mu_0$ of meridian, so that $\mu_0$ and $\lambda$ generate $H_1(\partial M)$ as an abelian group.  Then Dehn filling along $a\mu_0 + b\lambda$ produces a closed 3-manifold $Y_{a/b}$ with first homology $\ZZ/a\ZZ$, so for some orientations of $\gamma,\gamma'$ and integers $c,d$ we can write
\begin{align*}
\gamma &= p\mu_0 + c\lambda, &
\gamma' &= (p+1)\mu_0 + d\lambda.
\end{align*}
Letting $\mu = \mu_0 + (d-c)\lambda$ and $\epsilon=(p+1)c-pd$, it follows that
\begin{align*}
\gamma &= p\mu + \epsilon\lambda, &
\gamma' &= (p+1)\mu + \epsilon\lambda,
\end{align*}
and by definition we have $|\epsilon| = \Delta(\gamma,\gamma')$, which is $1$ by assumption.
\end{proof}

In other words, Lemma~\ref{lem:lp1f-meridian} yields a slope $\mu$ such that Dehn filling $M$ along $\mu$ produces a homology sphere $Y$ with core $K$ such that $Y_{\epsilon p}(K) = L(p,1)$ and $Y_{\epsilon(p+1)}(K) = L(p+1,1)$.  In fact, under the hypotheses of Theorem~\ref{thm:lp1-fillings} we claim that we can take $\epsilon = 1$.  

\begin{lemma} \label{lem:lp1f-epsilon}
If $Y$ is a homology sphere and $K \subset Y$ a knot such that $Y_{-p}(K) = L(p,1)$ and $Y_{-(p+1)}(K) = L(p+1,1)$ for some positive integer $p$, then $p \equiv 1 \pmod{12}$.
\end{lemma}

\begin{proof}
The two sides of $Y_{-p}(K) = L(p,1)$ have linking forms $-\frac{1}{p}$ and $\frac{1}{p}$, so these are equivalent and thus $-1$ is a square mod $p$.  It follows that $p$ cannot be a multiple of 4 or of any prime $q\equiv 3\pmod{4}$, since $-1$ is not a square modulo any of these numbers.  Moreover, $p$ cannot be $3\pmod{4}$, since it is a product of primes which are either $1\pmod{4}$ or 2, so $p$ must be either $1$ or $2$ modulo 4.  Since it is not a multiple of 3, it must also be either $1$ or $2$ modulo 3.  But the same holds true for $p+1$, since the linking forms $-\frac{1}{p+1}$ and $\frac{1}{p+1}$ are also equivalent, and so the only way this can be possible is if $p\equiv 1$ and $p+1\equiv 2$ modulo both 3 and 4.  In particular, $p \equiv 1 \pmod{12}$ as claimed.
\end{proof}

Given $Y_p(K) = L(p,1) = S^3_p(U)$ and likewise $Y_{p+1}(K) = S^3_{p+1}(U)$, \eqref{eq:cw-surgery-alternate} now implies that
\begin{equation}\label{eq:cw-pos-surgery}
\lambda(Y) + \frac{1}{p}A(K) = \lambda(Y) + \frac{1}{p+1}A(K) = 0,
\end{equation}
and thus $\lambda(Y) = A(K) = 0$.

\begin{proposition}
\label{prop:lp3f-invisible}
Let $Y$ be a homology sphere with $\lambda(Y) \geq 0$, and suppose for some knot $K \subset Y$ and integer $p\geq 2$ that $Y_p(K) = L(p,1)$ and $Y_{p+1}(K)$ is an L-space.  Then $Y$ is invisible, i.e.\ an L-space homology sphere with correction term $d(Y)=0$, and in fact $\lambda(Y)=0$.
\end{proposition}

\begin{proof}
Writing $Y_{p+1} = Y_{p+1}(K)$ for convenience, we use the surgery exact triangle
\[ \dots \xrightarrow{F_V^+} \hfp(Y) \xrightarrow{F_W^+} \hfp(L(p,1)) \to \hfp(Y_{p+1}) \to \dots \]
where $W$ and $V$ are the corresponding 2-handle cobordisms from $Y$ to $L(p,1)$ and from $Y_{p+1}$ to $Y$ respectively.  The latter two groups in the triangle have the form
\[ \hfp(L(p,1)) = \bigoplus_{i=0}^{p-1} \hfp(L(p,1),\spint_i) = \bigoplus_{i=0}^{p-1} \tower_{d(L(p,1),i)} \]
and $\hfp(Y_{p+1}) = \bigoplus_{i=0}^{p} \tower_{d(Y_{p+1},i)}$, with the subscripts denoting the grading of the bottom-most element $1$ in each tower $\tower = \F[U,U^{-1}] / (U\cdot \F[U])$.  In particular we have
\[ d(L(p,1),i) = \frac{(2i-p)^2}{4p} - \frac{1}{4}, \]
computed as in \cite{OzsvathGraded}.  We also know that $\hfp(Y) \cong \tower_{d(Y)} \oplus \hfred(Y)$, with $d(Y)$ an even integer and $\hfred(Y) = \hfp(Y)/\img(U^k)$ for $k \gg 0$.  Our goal is to show that $\hfred(Y)=0$ and $d(Y)=0$.

The map $F_W^+: \hfp(Y) \to \hfp(L(p,1))$ decomposes into summands $F^+_{W,\spinc}$ for each $\spinc \in \spc(W)$.  If we let $\Sigma \subset W$ be the result of capping off a Seifert surface for $K$ then $H_2(W) \cong \ZZ$ is generated by $[\Sigma]$, which has self-intersection $p$, and we can label the $\spc$ structures on $W$ as $\spinc_{n}$ ($n\in\ZZ$) such that $\langle c_1(\spinc_{n}), [\Sigma] \rangle = 2n+p$.  Then the $\spc$ structure $\spint_{n\!\!\pmod{p}} = \spinc_{n}|_{L(p,1)}$ is determined by $n\!\pmod{p}$.  
Now if $[D] \in H_2(W,\partial W)\cong \ZZ$ is the cocore of the 2-handle, then we must have $c_1(\spinc_{n}) = (2n+p)PD([D])$.  The natural map $H_2(W) \to H_2(W,\partial W)$ sends $(2n+p)[\Sigma]$ to $p(2n+p)[D]$, so $c_1(\spinc_{n})^2 = \frac{(2n+p)^2}{p^2} [\Sigma]^2 = \frac{(2n+p)^2}{p}$.  It follows from Theorem~\ref{thm:os-gradings} that
\[ \deg(F^+_{W,\spinc_{n}}) = \frac{(2n+p)^2}{4p} - \frac{5}{4}, \]
since $\chi(W)=\sigma(W)=1$.

If $x\in \hfp(Y)$ is a homogeneous element of even grading, then $F^+_{W,\spinc_{n}}(x)$ has grading $\operatorname{gr}(x) + \deg(F^+_{W,\spinc_{n}}) \equiv \frac{(2n+p)^2}{4p}-\frac{5}{4} \pmod{2}$.  On the other hand, each element of the codomain $\hfp(L(p,1), \spint_{n\!\!\pmod{p}})$ has grading congruent to
\[ d(L(p,1),\spint_{n\!\!\!\!\!\pmod{p}}) = \frac{(2n-p)^2}{4p} - \frac{1}{4} \equiv \left(\frac{(2n+p)^2}{4p}-\frac{5}{4}\right)+1 \pmod{2}, \]
and so $F^+_{W,\spinc_{n}}(x)$ must be zero.  Thus $x \in \ker(F^+_W)$.

Next, since every element of $\hfp(Y_{p+1})$ is in $\img(U^k)$ for all $k$, the same is true for $\img(F^+_V) \subset \hfp(Y)$.  By exactness the latter is equal to $\ker(F^+_W)$, hence all elements of $\hfp(Y)$ with even grading are in $\img(U^k)$.  It follows that $\hfred(Y)$ is supported entirely in odd gradings, and so by Theorem~\ref{thm:cw-sum-d} and the assumption that $\lambda(Y) \geq 0$, we have
\[ \frac{1}{2}d(Y) = \chi(\hfred(Y)) - \lambda(Y) = -\operatorname{rank}(\hfred(Y)) - \lambda(Y) \leq 0. \]

Now suppose that either $Y$ is not an L-space or $\lambda(Y) > 0$.  Then we have shown that $d(Y) < 0$, and so $Y$ is not the boundary of any smooth, negative definite 4-manifold by \cite[Corollary~9.8]{OzsvathGraded}.  If we take the positive definite 2-handle cobordism $W$ from $Y$ to $L(p,1)$, turn it upside down, and reverse its orientation, then we get a negative definite 2-handle cobordism $W'$ from $L(p,1)$ to $Y$, and $L(p,1)$ is the boundary of a negative definite linear plumbing $X$ of disk bundles over $S^2$, so the composition $Z = X \cup_{L(p,1)} W'$ is negative definite with $\partial Z = Y$.  This is a contradiction, so $Y$ must in fact be an L-space with $\lambda(Y) = 0$, and we have $d(Y) = -2(\operatorname{rank}(\hfred(Y))+\lambda(Y)) = 0$.
\end{proof}

\begin{proof}[Proof of Theorem~\ref{thm:lp1-fillings}]
Using Lemmas~\ref{lem:lp1f-meridian} and \ref{lem:lp1f-epsilon}, we have a homology sphere $Y$ obtained from $M$ by Dehn filling with core $K$ so that $Y_p(K) = L(p,1)$ and $Y_{p+1}(K) = L(p+1,1)$.  It follows from \eqref{eq:cw-pos-surgery} that $\lambda(Y)=0$, and so Proposition~\ref{prop:lp3f-invisible} ensures that $Y$ is an invisible manifold.  Since $Y$ is invisible and $Y_p(K) = L(p,1)$, it follows from Corollary~\ref{cor:invisible-lens} that $Y=S^3$ and $K$ is the unknot, and so the exterior $M$ of $K$ is a solid torus as desired.
\end{proof}

\section{Knots in $\rpthree$ with distance-one lens space fillings of order 5}
\label{sec:surgery25}

In this section we address the question of when $\rpthree$ and a lens space of the form $L(5,q)$ can belong to a triad.  

\begin{theorem} \label{thm:rp3-l5q}
Let $M$ be a manifold with torus boundary for which a pair of distance-one Dehn fillings produce $\rpthree$ and a lens space $L(5,q)$.  Then $M$ is a solid torus.
\end{theorem}

Theorem~\ref{thm:rp3-l5q} is the final result we will need in order to prove Theorem~\ref{thm:formal-l-spaces}.  In particular, this will be a critical part of classifying formal L-spaces with determinant 7.  

We begin by determining which lens spaces of order 5 can occur.

\begin{proposition} \label{prop:25-lens}
Let $M$ and $L(5,q)$ satisfy the hypotheses of Theorem~\ref{thm:rp3-l5q}.  Then $M$ admits a Dehn filling $Y$ with core $K$ such that $Y$ is a homology sphere, $Y_2(K) = \rpthree$, and $Y_{5/p}(K) = L(5,p)$ where $p$ is either 2 or 3.  Moreover, every 3-manifold of the form $Y_{a/b}(K)$ has Casson-Walker invariant equal to $\lambda(S^3_{a/b}(U))$.
\end{proposition}

\begin{proof}
We know that $M$ is a homology $S^1 \times D^2$, since otherwise the core of the $\rpthree$ filling is nullhomologous and so any other filling would have homology of even order.  We identify curves $\mu,\lambda \subset \partial M$ such that $\lambda$ generates the kernel of the natural map $H_1(\partial M) \to H_1(M)$ and $\mu\cdot\lambda=1$.  For some odd integer $n$, Dehn filling along the curve $\gamma = 2\mu + n\lambda$ produces the $\rpthree$ filling.  If we write $n=2k+1$ and set $\mu'=\mu+k\lambda$ then we have $\gamma = 2\mu'+\lambda$ and $\mu'\cdot\lambda=1$.  Thus Dehn filling along $\mu'$ produces the desired homology sphere $Y$ and core $K$ such that $Y_2(K) = \rpthree$.   The $L(5,q)$ filling of $M$ is at distance one from $2\mu'+\lambda$ by assumption, so it must be along $5\mu'+p\lambda$ where $p$ is either 2 or 3, i.e.\ $Y_{5/p}(K) = L(5,q)$.  Since $L(5,2) = L(5,3)$, we would like to see that $Y_{5/p}(K)$ is not homeomorphic to $\pm L(5,1)$.  However, since $p = 2$ or $3$, the linking form rules this out, and we can take $p = q$. 

For the second claim, we apply the Casson-Walker surgery formula \eqref{eq:cw-surgery-alternate} to $Y_2(K)$ and $Y_{5/p}(K)$:
\begin{align*}
\lambda(Y_2(K)) - \lambda(\rpthree) &= \lambda(Y) + \frac{1}{2}A(K) \\
\lambda(Y_{5/p}(K)) - \lambda(L(5,p)) &= \lambda(Y) + \frac{p}{5}A(K).
\end{align*}
The left sides of both equations are zero, and we obtain $A(K) = \lambda(Y) = 0$.  The surgery formula then says that $\lambda(Y_{a/b}(K)) = \lambda(S^3_{a/b}(U))$ for all $\frac{a}{b}$.
\end{proof}

We claim that it suffices to consider the case $Y_2(K) = \rpthree$ and $Y_{5/3}(K) = L(5,3)$ (i.e.\ $p=3$) for now.  Let's see how this implies the remaining case $p=2$ of Proposition~\ref{prop:25-lens}, i.e.\ $Y_2(K) = L(5,2)$.  Following the proof, we could instead take $\mu^* = \mu+(k+1)\lambda = \mu'+\lambda$ and define $Y^*$ to be the homology sphere attained by filling along $\mu^*$ with core $K^*$.  We have $\gamma=2\mu^*-\lambda$, so $Y^*_{-2}(K^*) = \rpthree$.  Then $5\mu'+2\lambda = 5\mu^*-3\lambda$, so that $Y^*_{-5/3}(K^*) = L(5,2) = -L(5,3)$.  If we reverse orientation, then we get $(-Y^*)_2(-K^*) = \rpthree$ and $(-Y^*)_{5/3}(-K^*) = L(5,3)$.  Thus, the $p=3$ case would imply that the exterior of $-K^*$ in $-Y^*$, which is orientation-reversing homeomorphic to $M$, is a solid torus, completing the proof.

With the preceding understood, we now suppose for the remainder of this section that $Y_2(K) = \rpthree$ and $Y_{5/3}(K) = L(5,3)$.  In what follows we will write $Y_{a/b} = Y_{a/b}(K)$ for convenience.

\begin{lemma} \label{lem:shortexactsurgery}
For all $n \geq 1$, there is a short exact sequence 
\[ 0 \to \hfp(Y_{(2n-1)/n}) \to \hfp(Y_{(2n+1)/(n+1)}) \to \hfp(\rpthree) \to 0, \]
in which the two non-trivial maps are induced by negative definite 2-handle cobordisms.
\end{lemma}

\begin{proof}
These invariants fit into the surgery exact triangle \eqref{eq:2-exact-triangle}, so it will suffice to show that if $W$ is the 2-handle cobordism from $Y_2 = \rpthree$ to $Y_{(2n-1)/n}$, corresponding to attaching a handle along an $n$-framed meridian of $K$ (after a 2-framed surgery on $K\subset Y$), then the induced map $F^+_W$ is zero.  By the discussion above Lemma~\ref{lem:d-2p-Z}, we see that $W$ is positive definite for $n \geq 1$.  Since $\rpthree$ is an L-space, $F^+_{W,\spinc} = 0$ for all $\spinc \in \spc(W)$ by Lemma~\ref{lem:surgery-fplus}.
\end{proof}

\begin{proposition}\label{prop:two-five-l-space}
If $Y_2 = \rpthree$ and $Y_{5/3} = L(5,3)$, then $Y_{3/2}$ is an L-space.
\end{proposition}

\begin{proof}
The case $n=2$ of Lemma~\ref{lem:shortexactsurgery} produces a short exact sequence
\[ 0 \to \hfp(Y_{3/2}) \xrightarrow{F^+_V} \hfp(L(5,3)) \xrightarrow{F^+_W} \hfp(\rpthree) \to 0 \]
corresponding to surgeries on $K \subset Y$ with slope $\frac{3}{2}$, $\frac{5}{3}$, and $\frac{2}{1}$ respectively, and $V$ and $W$ are each negative definite 2-handle cobordisms.  In particular, since $W$ is negative definite and $\rpthree$ is an L-space, Lemma~\ref{lem:surgery-fplus} says that each $F^+_{W,\spinc}$ must be surjective.  In fact, since both $L(5,3)$ and $\rpthree$ are L-spaces, $F^+_{W,\spinc}$ is a surjective, $U$-equivariant map of the form
\[ F^+_{W,\spinc}: \tower_{d(L(5,3), \spinc|_{L(5,3)})} \to \tower_{d(\rpthree, \spinc|_{\rpthree})}, \]
so it is determined entirely by its grading $\gr(F^+_{W,\spinc}) = \frac{1}{4}\big(c_1(\spinc)^2 - 2\chi(W) - 3\sigma(W)\big)$, and moreover the $\spc$ structures on either side satisfy 
\begin{equation} \label{eq:drp3-dl53}
d(\rpthree,\spinc|_{\rpthree}) - d(L(5,3),\spinc|_{\rpthree}) \equiv \gr(F^+_{W,\spinc}) = \frac{c_1(\spinc)^2 + 1}{4} \pmod{2}.
\end{equation}

The $\spc$ structures on $L(5,3)$ have $d$-invariants $-\frac{2}{5},-\frac{2}{5},0,\frac{2}{5},\frac{2}{5}$ in some order, and those on $\rpthree$ have $d$-invariants $-\frac{1}{4}$ and $\frac{1}{4}$.  It is not hard to check that we can write $c_1(\spinc)^2 = -\frac{k^2}{10}$ for some integer $k$ whose parity is fixed.  If $d(\rpthree, \spinc|_{\rpthree}) = \frac{1}{4}$ and $d(L(5,3),\spinc|_{L(5,3)}) = -\frac{2}{5}, 0, \frac{2}{5}$ then $k^2 \equiv 64, 0, 16 \pmod{80}$ respectively by \eqref{eq:drp3-dl53}, hence $k\pmod{20}$ belongs to $\{8,12\}$, $\{0\}$, or $\{4,16\}$ respectively.  If instead $d(\rpthree,\spinc|_{\rpthree}) = -\frac{1}{4}$ then we have $k^2 \equiv 4, 20, 36 \pmod{80}$ and so $k \pmod{20}$ belongs to $\{2,18\}$, $\{10\}$, or $\{6,14\}$ respectively.  Since the restrictions of $\spinc$ to $L(5,3)$ and to $\rpthree$ are determined by $k\pmod{10}$ and $k\pmod{4}$, the even residue classes mod 20 are in bijection with elements of $\spc(L(5,3)) \times \spc(\rpthree)$, and although this bijection is not quite unique the pair of $d$-invariants associated to each residue class is uniquely determined.  In particular, this says that $F^+_W$ is uniquely determined as a map
\[ \tower_{0} \oplus (\tower_{-2/5})^{\oplus 2} \oplus (\tower_{2/5})^{\oplus 2} \to \tower_{1/4} \oplus \tower_{-1/4} \]
of $\F[U]$-modules, up to possibly reordering the $\tower_{-2/5}$ summands and the $\tower_{2/5}$ summands.

The above argument uses only the fact that $W$ is a negative definite 2-handle cobordism, so it applies equally well if we replace $K \subset Y$ with the unknot $U \subset S^3$ to get another cobordism $W'$ from $L(5,3)$ to $\rpthree$.  There is a natural identification $\spc(W) \cong \spc(W')$ which preserves these gradings and the $d$-invariants of the restrictions to either boundary component, so up to possibly reordering $\spc(L(5,3))$ in a way which preserves the corresponding $d$-invariants, we conclude that the maps $F^+_W$ and $F^+_{W'}$ are equal.

Now $F^+_{W'}$ fits into a surgery exact sequence of $\F[U]$-modules of the form
\[ 0 \to \hfp(L(3,2)) \xrightarrow{F^+_{V'}} \hfp(L(5,3)) \xrightarrow{F^+_{W'}} \hfp(\rpthree) \to 0. \]
Since $L(3,2)$ is an L-space, the $U$ action on $\hfp(L(3,2))$ is surjective and it follows that $\img(F^+_{V'}) \subset U\cdot \img(F^+_{V'})$.  But $\img(F^+_{V'}) = \ker(F^+_{W'})$ by exactness, and this is equal to $\ker(F^+_W) = \img(F^+_V)$ since $F^+_W = F^+_{W'}$, so $\img(F^+_V) \subset U\cdot \img(F^+_V)$ as well.  Given any element $a \in \hfp(Y_{3/2})$ we have $F^+_V(a) = U \cdot F^+_V(b)$ for some $b$, and $F^+_V$ is injective so $a=Ub$.  Thus $U: \hfp(Y_{3/2}) \to \hfp(Y_{3/2})$ is surjective, and we conclude that $Y_{3/2}$ is an L-space as well.
\end{proof}

It remains to compute the $d$-invariants of $Y_{3/2}$.  By Theorem~\ref{thm:cw-sum-d}, their sum satisfies
\[ 3\lambda(Y_{3/2}) = \sum_{\spint \in \spc(Y_{3/2})} \left(\chi(\hfred(Y_{3/2},\spint)) - \frac{1}{2}d(Y_{3/2},\spint)\right), \]
and from Proposition~\ref{prop:25-lens} we know that the left side equals $3\lambda(L(3,2))$; this is $\frac{1}{12}$ by another application of Theorem~\ref{thm:cw-sum-d} to $L(3,2)$, an L-space whose $d$-invariants are $\{-\frac{1}{2},\frac{1}{6},\frac{1}{6}\}$.  By Proposition~\ref{prop:two-five-l-space}, $Y_{3/2}$ is also an L-space, so we have
\[ \sum_{\spint \in \spc(Y_{3/2})} d(Y_{3/2},\spint) = -\frac{1}{6}. \]
Let $\spint_0,\spint_1,\spint_2$ denote the $\spc$ structures on $Y_{3/2}$, with $\spint_0$ the unique spin structure and $\spint_2 = \overline{\spint_1}$, and let $d_i = d(Y_{3/2},\spint_i)$; we note that $d_1 = d_2$.

\begin{proposition}
The $d$-invariants of $Y_{3/2}$ are $d_0 = -\frac{1}{2}$ and $d_1 = d_2 = \frac{1}{6}$.  In other words, $Y_{3/2}$ is a Heegaard Floer $L(3,2)$.
\end{proposition}

\begin{proof}
By reversing the orientation of the 2-handle attachment $X$ from $\rpthree$ to $Y_{3/2}$, we obtain a negative definite cobordism $-X$ from $Y_{3/2}$ to $\rpthree$.  For any Spin$^c$ structure $\spinc$ on $-X$, we see from Lemma~\ref{lem:surgery-fplus} that $F^+_{-X,\spinc}: HF^+(Y_{3/2},\spinc|_{Y_{3/2}}) \to HF^+(\rpthree, \spinc|_{\rpthree})$ is surjective.  Therefore, since $c_1(\spinc)^2 \leq 0$ it follows that 
\[
d(\rpthree, \spinc|_{\rpthree}) - d(Y_{3/2},\spinc|_{Y_{3/2}})  \leq \frac{c_1(\spinc)^2 - 2\chi(-X) - 3\sigma(-X)}{4} \leq \frac{1}{4}.    
\]
Since $d(\rpthree, \spinc|_{\rpthree}) = -\frac{1}{4}$ or $\frac{1}{4}$, we see that $d(Y_{3/2},\spinc|_{Y_{3/2}}) \geq -\frac{1}{2}$ for each $\spinc$.  It is straightforward to check that each $\spc$ structure on $Y_{3/2}$ extends over $-X$, so we see that $d(Y_{3/2}, \spint) \geq -\frac{1}{2}$ for all $\spint$ on $Y_{3/2}$.  

By Lemma~\ref{lem:d-2p-Z}, the set of $d$-invariants for $Y_{3/2}$ agrees mod 2 with that of $L(3,2)$, since the linking form of $Y_{3/2}$ is $\frac{2}{3}$.  However, the $d$-invariants of $L(3,2)$ are $\{-\frac{1}{2}, \frac{1}{6}, \frac{1}{6}\}$.  We can conclude that the three $d$-invariants of $Y_{3/2}$ are each at least $-\frac{1}{2}$, agree with the set $\{ -\frac{1}{2}, \frac{1}{6}, \frac{1}{6} \}$ mod 2, and sum to $-\frac{1}{6}$.  This implies that the $d$-invariants of $Y_{3/2}$ are in fact as claimed. 
\end{proof}

\begin{proof}[Proof of Theorem~\ref{thm:rp3-l5q}]
Now we know that $Y_2 = \rpthree$ and that $Y_{3/2}$ is a Heegaard Floer $L(3,2)$.  We first show that $Y_1(K)$ is an invisible 3-manifold.  Indeed, we have a short exact sequence
\[ 0 \to \hfp(Y_1) \to \hfp(Y_{3/2}) \to \hfp(\rpthree) \to 0 \]
by Lemma~\ref{lem:shortexactsurgery}, with the middle two maps coming from negative definite 2-handle cobordisms.  Arguing as in Proposition~\ref{prop:two-five-l-space}, since $Y_{3/2}$ is an L-space with the same $d$-invariants as $L(3,2)$, the grading of each $F^+_{W,\spinc}: \hfp(Y_{3/2},\spinc|_{Y_{3/2}}) \to \hfp(\rpthree,\spinc|_{\rpthree})$ mod 2 uniquely determines the $d$-invariants of $\spinc|_{Y_{3/2}}$ and $\spinc|_{\rpthree}$.  Thus the cobordism map $F^+_W: \hfp(Y_{3/2}) \to \hfp(\rpthree)$ is once again determined entirely by the gradings in each $\spc$ structure, in exactly the same way as the map $\hfp(L(3,2)) \to \hfp(\rpthree)$ arising from the exact triangle for surgeries on the unknot $U\subset S^3$:
\[
\ldots \to \hfp(S^3_1(U)) \to \hfp(S^3_{3/2}(U)) \to \hfp(S^3_2(U)) \to \ldots
\] 
Since $1$-surgery on $U \subset S^3$ is again an L-space (namely $S^3$), we conclude exactly as in the case of $Y_{3/2}$ in Proposition~\ref{prop:two-five-l-space} that $Y_1$ is also an L-space.  Proposition~\ref{prop:25-lens} says that $\lambda(Y_1) = 0$, so $d(Y_1) = 2(\chi(\hfred(Y_1)) - \lambda(Y_1)) = 0$ and thus $Y_1$ is invisible.

Let $Y' = Y_1(K)$ be the result of Dehn filling the exterior $M$ of $K$ along the curve $\mu+\lambda$, with core $K'$.  Then since $2\mu+\lambda = 2(\mu+\lambda) - \lambda$ we have $\rpthree = Y_2(K) = Y'_{-2}(K')$, so $Y'$ is an invisible 3-manifold satisfying $Y'_{-2}(K') = S^3_{-2}(U)$.  Corollary~\ref{cor:invisible-lens} now says that $Y'=S^3$ and $K'=U$, and since the exterior of $K' \subset Y'$ is also $M$ we conclude that $M$ is a solid torus, as desired.
\end{proof}


\section{The proof of Theorem~\ref{thm:formal-l-spaces}}\label{sec:classification}

We divide the proof of Theorem~\ref{thm:formal-l-spaces} into several parts, based on the value of the determinant of the three-manifold $Y$ in question.  In each case we identify $Y$ as part of a triad $(Y,Y_0,Y_1)$ with $\det(Y) = \det(Y_0) + \det(Y_1)$, and use the possible values of $Y_0$ and $Y_1$ to determine the possible values of $Y$.  Since the order of the elements in a triad does not matter, we may assume without loss of generality that $\det(Y_0) \leq \det(Y_1)$ in each of the arguments below.  We will label each case ``$a+b=c$'' to indicate that $\det(Y_0)=a$, $\det(Y_1)=b$, and $\det(Y)=c$.

\subsection{The case $1+n=(n+1)$}

Since the cases $1+3=4$, $1+4=5$, $1+5=6$, and $1+6=7$ are all essentially the same, we combine them here to avoid repetition.

\begin{proposition} \label{prop:1-plus-n}
Suppose that $Y$ is a formal L-space with $\det(Y)=n+1$ for some $n \leq 6$, and that all formal L-spaces with determinant $n$ are connected sums of lens spaces.  If $(Y,Y_0,Y_1)$ is a triad of formal L-spaces with $\det(Y_0)=1$ and $\det(Y_1)=n$, then $Y$ is also a connected sum of lens spaces.
\end{proposition}

\begin{proof}
By assumption we have $Y_0 = S^3$, so $Y$ and $Y_1$ are the results of surgeries with consecutive integer slopes on some knot $J_0 \subset S^3$.  Since $H_1(Y_1)$ is cyclic we cannot have $Y_1 = \rpthree \# \rpthree$, so either $Y_1$ is a lens space of order $n \leq 6$ or it is $\pm L(3,1) \# \rpthree$.

If $Y_1$ is a lens space, then Theorem~\ref{thm:kmos} says that $J_0$ is either the unknot or a trefoil.  If it is the unknot, then certainly $Y$ is a lens space as well.  Otherwise $Y_1$ is a lens space resulting from a $\pm n$-surgery on a trefoil, and $n\leq 6$, so $Y_1$ is either $+5$-surgery on the right handed trefoil $T_{2,3}$ or $-5$-surgery on the left handed trefoil $-T_{2,3}$ \cite{Moser}.  Then $Y$ is $\pm 6$-surgery on $\pm T_{2,3}$, and this is $\pm L(3,1) \# \rpthree$.

If instead we have $Y_1 = \pm L(3,1) \# \rpthree$, then \cite[Theorem~1.5]{GreeneCabling} says that $J_0$ must be $\pm 6$-surgery on $\pm T_{2,3}$, so $Y$ is the result of $\pm 7$-surgery on $\pm T_{2,3}$, and this is $\pm L(7,4)$.
\end{proof}

\begin{remark}
The same argument applies for $1+7=8$ to show that $J_0$ is either an unknot or a trefoil; in this case $Y$ would be either a lens space or the small Seifert fibered space $S^3_{\pm 8}(\pm T_{2,3})$.
\end{remark}

\subsection{Formal L-spaces of determinant 4}
\begin{theorem}\label{thm:qa4}
Let $Y$ be a formal L-space with $\det(Y)=4$.  Then $Y$ is either $\pm L(4,1)$ or $\rpthree\#\rpthree$.
\end{theorem}

\begin{proof}
We suppose that $Y$ belongs to a triad $(Y,Y_0,Y_1)$ with $\det(Y_0)+\det(Y_1)=4$ and $\det(Y_0) \leq \det(Y_1)$ as above.  By Proposition~\ref{prop:1-plus-n}, we need only consider the case $2+2=4$, in which $Y_0=Y_1=\rpthree$ by Theorem~\ref{thm:greene-classification}.  We then have a knot $J_0 \subset Y_0 = \rpthree$ with a non-trivial $Y_1 = \rpthree$ surgery.  Since the nontrivial filling has to be distance one from the trivial filling, Corollary~\ref{cor:lens-cosmetic} says that $J_0$ is an unknot in $\rpthree$.  Since $\det(Y) = 4$, we have $H_1(Y) = \Z/2\Z \oplus \Z/2\Z$.  It follows that $Y = \rpthree \# \rpthree$, and this completes the proof.  
\end{proof}

\subsection{Formal L-spaces of determinant 5}
\begin{theorem}\label{thm:qa5}
Let $Y$ be a formal L-space with $\det(Y)=5$.  Then $Y$ is a lens space.
\end{theorem}

\begin{proof}
Suppose that $Y$ belongs to a triad $(Y,Y_0,Y_1)$ of formal L-spaces with $\det(Y_0) + \det(Y_1) = \det(Y) = 5$ and $\det(Y_0) \leq \det(Y_1)$.  In the case $1+4=5$, Proposition~\ref{prop:1-plus-n} says that $Y$ is a lens space, so we need only consider the case $2+3=5$, in which case Theorem~\ref{thm:greene-classification} says that $Y_0 = \rpthree$ and $Y_1 = \pm L(3,1)$.

The three members of the triad $(Y,Y_0,Y_1)$ all arise as Dehn fillings of the same manifold $M$ along slopes with pairwise distance 1.  In particular, either $M$ or $-M$ has two distance-1 slopes with Dehn fillings $L(2,1)$ and $L(3,1)$, where the sign is determined by whether $Y_1 = L(3,1)$ or $Y_1 = -L(3,1)$, so by Theorem~\ref{thm:lp1-fillings} we know that $\pm M$ and hence $M$ itself is $S^1 \times D^2$.  Since $Y$ is another Dehn filling of $M$, it must be a lens space $L(5,q)$.
\end{proof}

\subsection{Formal L-spaces of determinant 6}

\begin{theorem} \label{thm:det-6}
Let $Y$ be a formal L-space with $\det(Y)=6$.  Then $Y$ is either $\pm L(6,1)$ or $\pm L(3,1)\#\rpthree$.
\end{theorem}

Let $Y$ belong to a triad $(Y,Y_0,Y_1)$ of formal L-spaces with $\det(Y_0) + \det(Y_1) = \det(Y) = 6$ and $\det(Y_0) \leq \det(Y_1)$.  Again, Proposition~\ref{prop:1-plus-n} proves this theorem in the case $1+5=6$, so we need only consider two cases: $2+4=6$ and $3+3=6$.

\begin{proposition}
If $\det(Y_0)=2$ and $\det(Y_1) = 4$, then $Y = \pm L(3,1) \# \rpthree$.
\end{proposition}

\begin{proof}
We must have $Y_0 = \rpthree$, and $Y_1$ is either $\pm L(4,1)$ or $\rpthree \# \rpthree$  by Theorem~\ref{thm:qa4}.  Let $J_0 \subset \rpthree$ be the knot on which surgery produces $Y_1$ and $Y$.  A filling of its exterior at distance one from the trivial filling produces $Y_1$, and $\gcd(|H_1(Y_0)|,|H_1(Y_1)|)=2$, so $J_0$ cannot be primitive by Lemma~\ref{lem:distance-one}; thus $J_0$ is nullhomologous.  

Now since some $p$-surgery on $J_0$ produces $Y_1$ we have $H_1(Y_1) = \Z/2\Z \oplus \Z/p\Z$, and in particular it cannot be $\Z/4\Z = H_1(\pm L(4,1))$.  We must have $Y_1 = \rpthree \# \rpthree$ and $p=\pm2$.  But then $p$-surgery on $J_0$ is homeomorphic to $p$-surgery on the unknot in $\rpthree$, so by Theorem~\ref{thm:surgery-characterization}, $J_0$ must be the unknot.  Therefore $Y$ is the result of $\pm 3$-surgery on the unknot in $\rpthree$, namely $\pm L(3,1) \# \rpthree$, as claimed.
\end{proof}

\begin{proposition}
If $\det(Y_0) = \det(Y_1) = 3$, then $Y$ is $\pm L(3,1) \# \rpthree$.
\end{proposition}

\begin{proof}
We must have $Y_0 = \pm L(3,1)$ and $Y_1 = \pm L(3,1)$, though not necessarily with the same orientation.  Again, we identify $J_0$ in $Y_0$ for which $Y_1$ and $Y$ are obtained by distance one surgeries.  Applying Lemma~\ref{lem:distance-one} to $Y_1$, we see that $J_0$ must be nullhomologous.  If $Y_0$ and $Y_1$ have the same orientation, then Theorem~\ref{thm:surgery-characterization} says that $J_0$ is the unknot, and then $\det(Y)=6$ implies that $Y = Y_0 \# \rpthree = \pm L(3,1) \# \rpthree$.

In the remaining case we have $Y_0$ homeomorphic to $-Y_1$, and we claim that this can never happen.  Indeed, in this case $Y_1$ must be $\pm 1$-surgery on $J_0 \subset Y_0$, since these are the only distance-one surgeries which preserve the order of $H_1$.  Since $J_0$ is nullhomologous, this contradicts \cite[Corollary 3.10]{CochranGergesOrr}.  
\end{proof}

\subsection{Formal L-spaces of determinant 7}

\begin{theorem}\label{thm:qa7}
Let $Y$ be a formal L-space with $\det(Y)=7$.  Then $Y$ is a lens space.
\end{theorem}

Once again, we let $Y$ belong to a triad of formal L-spaces $(Y,Y_0,Y_1)$ with $\det(Y_0)+\det(Y_1)=\det(Y)=7$ and $\det(Y_0) \leq \det(Y_1)$.  In the case $1+6=7$, Proposition~\ref{prop:1-plus-n} says that $Y$ must be a lens space; we wish to prove the same conclusion for $2+5=7$ and $3+4=7$.

If $J_0 \subset Y_0$ is a knot whose exterior produces each of these 3-manifolds through pairwise distance-one fillings, we note that $J_0$ must generate $H_1(Y_0)$, since $\det(Y_0)$ and $\det(Y_1)$ are relatively prime.  In the case $2+5=7$, we have $Y_0 = \rpthree$ and $Y_1 = L(5,q)$ for some $q$ by Theorem~\ref{thm:qa5}, so Theorem~\ref{thm:rp3-l5q} says that the exterior of $J_0$ is a solid torus, hence $Y$ is a lens space.  Thus the only case which remains to be proved is $3+4=7$.

\begin{proposition}
If $\det(Y_0) = 3$ and $\det(Y_1) = 4$, then $Y = L(7,q)$ for some $q$.
\end{proposition}

\begin{proof}
We have $Y_0 = \pm L(3,1)$ and $J_0 \subset Y_0$ is primitive, so any Dehn surgery on $J_0$ must have cyclic first homology.  This means that $Y_1$ cannot be $\rpthree \#\rpthree$, so we have $Y_1 = \pm L(4,1)$ by Theorem~\ref{thm:qa4}.

If $Y_0 = L(3,1)$ and $Y_1 = L(4,1)$, then the exterior $M$ of $J_0$ must be a solid torus by Theorem~\ref{thm:lp1-fillings}, so $Y$ is necessarily a lens space.  The same argument applies if $Y_0 = -L(3,1)$ and $Y_1 = -L(4,1)$, since $L(3,1)$ and $L(4,1)$ are then a pair of distance-one fillings of $-M$.

In the remaining cases we have $Y_0 = \pm L(3,1)$ and $Y_1 = \mp L(4,1)$, and we claim that this is impossible.  After possibly reversing orientation, we can assume that $Y_0 = L(3,1)$ and $Y_1 = -L(4,1)$.  As in the proof of Lemma~\ref{lem:lp1f-meridian}, we can choose a meridian $\mu_0$ on the exterior $M$ of $J_0$ so that $Y_0$ is the result of Dehn filling along a curve $\gamma_0 = 3\mu_0 + n\lambda$ for some $n$ of the form $3k \pm 1$, with $\lambda$ the rational longitude.  Using the meridian $\mu = \mu_0 + k\lambda$, we have $\gamma_0 = 3\mu \pm \lambda$.  If $Y_1$ results from Dehn filling along $\gamma_1$, then since $|H_1(Y_1)|=4$ and $\Delta(\gamma_0,\gamma_1)=1$ we must have $\gamma_1=4\mu \pm \lambda$.  In particular, if $Z$ is the homology sphere resulting from Dehn filling along $\mu$ with core $K\subset Z$, then we have $Y_0 = Z_{\pm 3}(K)$ and $Y_1 = Z_{\pm 4}(K)$.
  
From this, we see that $Y_0$ (respectively $Y_1$) has linking form $\pm \frac{1}{3}$ (respectively $\pm \frac{1}{4}$).  Since $Y_0 = L(3,1)$ has linking form $+\frac{1}{3}$, which is not equivalent to $-\frac{1}{3}$, we see that $Y_1$ must have linking form $+\frac{1}{4}$.  However, this is not equivalent to $-\frac{1}{4}$, the linking form of $-L(4,1) = Y_1$.  This yields the desired contradiction.
\end{proof}

\bibliographystyle{alpha}
\bibliography{References}

\begin{thebibliography}{KMOS07}

\bibitem[BL90]{BoyerLines}
Steven Boyer and Daniel Lines.
\newblock Surgery formulae for {C}asson's invariant and extensions to homology
  lens spaces.
\newblock {\em J. Reine Angew. Math.}, 405:181--220, 1990.

\bibitem[CGO01]{CochranGergesOrr}
Tim~D. Cochran, Amir Gerges, and Kent Orr.
\newblock Dehn surgery equivalence relations on 3-manifolds.
\newblock {\em Math. Proc. Cambridge Philos. Soc.}, 131(1):97--127, 2001.

\bibitem[CK09]{CK}
Abhijit Champanerkar and Ilya Kofman.
\newblock Twisting quasi-alternating links.
\newblock {\em Proc. Amer. Math. Soc.}, 137(7):2451--2458, 2009.

\bibitem[Gai14]{Gainullin}
Fyodor Gainullin.
\newblock {Only finitely many alternating knots can yield a given manifold by
  surgery}.
\newblock arXiv 1411.2249, 2014.

\bibitem[Gai15]{Gainullin-complement}
Fyodor Gainullin.
\newblock {Heegaard Floer homology and knots determined by their complements}.
\newblock arXiv:1504.06180, 2015.

\bibitem[GL14]{GreeneLevine}
Joshua~Evan Greene and Adam~Simon Levine.
\newblock Strong {H}eegaard diagrams and strong {L}-spaces.
\newblock arXiv:1411.6329, 2014.

\bibitem[Gre10]{GreeneNonQA}
Joshua Greene.
\newblock Homologically thin, non-quasi-alternating links.
\newblock {\em Math. Res. Lett.}, 17(1):39--49, 2010.

\bibitem[Gre15]{GreeneCabling}
Joshua~Evan Greene.
\newblock L-space surgeries, genus bounds, and the cabling conjecture.
\newblock {\em J. Differential Geom.}, 100(3):491--506, 2015.

\bibitem[HR85]{HR}
Craig Hodgson and J.~H. Rubinstein.
\newblock Involutions and isotopies of lens spaces.
\newblock In {\em Knot theory and manifolds ({V}ancouver, {B}.{C}., 1983)},
  volume 1144 of {\em Lecture Notes in Math.}, pages 60--96. Springer, Berlin,
  1985.

\bibitem[JK13]{JongKishimoto}
In~Dae Jong and Kengo Kishimoto.
\newblock On positive knots of genus two.
\newblock {\em Kobe J. Math.}, 30(1-2):1--18, 2013.

\bibitem[KM07]{KMbook}
Peter Kronheimer and Tomasz Mrowka.
\newblock {\em Monopoles and three-manifolds}, volume~10 of {\em New
  Mathematical Monographs}.
\newblock Cambridge University Press, Cambridge, 2007.

\bibitem[KMOS07]{KMOS}
Peter Kronheimer, Tomasz Mrowka, Peter Ozsv{\'a}th, and Zolt{\'a}n Szab{\'o}.
\newblock Monopoles and lens space surgeries.
\newblock {\em Ann. of Math. (2)}, 165(2):457--546, 2007.

\bibitem[KT80]{KimTollefson}
Paik~Kee Kim and Jeffrey~L. Tollefson.
\newblock Splitting the {PL} involutions of nonprime {$3$}-manifolds.
\newblock {\em Michigan Math. J.}, 27(3):259--274, 1980.

\bibitem[LP14]{LackenbyPurcell}
Marc Lackenby and Jessica~S. Purcell.
\newblock {Cusp volumes of alternating knots}.
\newblock arXiv 1410.6297, 2014.

\bibitem[MO08]{ManolescuOzsvath}
Ciprian Manolescu and Peter Ozsv{\'a}th.
\newblock On the {K}hovanov and knot {F}loer homologies of quasi-alternating
  links.
\newblock In {\em Proceedings of {G}\"okova {G}eometry-{T}opology {C}onference
  2007}, pages 60--81. G\"okova Geometry/Topology Conference (GGT), G\"okova,
  2008.

\bibitem[Mos71]{Moser}
Louise Moser.
\newblock Elementary surgery along a torus knot.
\newblock {\em Pacific J. Math.}, 38:737--745, 1971.

\bibitem[Ni06]{NiNote}
Yi~Ni.
\newblock A note on knot {F}loer homology of links.
\newblock {\em Geom. Topol.}, 10:695--713, 2006.

\bibitem[OS03]{OzsvathGraded}
Peter Ozsv{\'a}th and Zolt{\'a}n Szab{\'o}.
\newblock Absolutely graded {F}loer homologies and intersection forms for
  four-manifolds with boundary.
\newblock {\em Adv. Math.}, 173(2):179--261, 2003.

\bibitem[OS04]{Ozsvath2004a}
Peter Ozsv{\'a}th and Zolt{\'a}n Szab{\'o}.
\newblock Holomorphic disks and three-manifold invariants: properties and
  applications.
\newblock {\em Ann. of Math. (2)}, 159(3):1159--1245, 2004.

\bibitem[OS05]{OzsvathBranched}
Peter Ozsv{\'a}th and Zolt{\'a}n Szab{\'o}.
\newblock On the {H}eegaard {F}loer homology of branched double-covers.
\newblock {\em Adv. Math.}, 194(1):1--33, 2005.

\bibitem[OS06]{OzsvathSmooth}
Peter Ozsv{\'a}th and Zolt{\'a}n Szab{\'o}.
\newblock Holomorphic triangles and invariants for smooth four-manifolds.
\newblock {\em Adv. Math.}, 202(2):326--400, 2006.

\bibitem[OS11]{OzsvathRational}
Peter~S. Ozsv{\'a}th and Zolt{\'a}n Szab{\'o}.
\newblock Knot {F}loer homology and rational surgeries.
\newblock {\em Algebr. Geom. Topol.}, 11(1):1--68, 2011.

\bibitem[RG72]{RademacherGrosswald}
Hans Rademacher and Emil Grosswald.
\newblock {\em Dedekind sums}.
\newblock The Mathematical Association of America, Washington, D.C., 1972.
\newblock The Carus Mathematical Monographs, No. 16.

\bibitem[Rus04]{Rustamov}
Raif Rustamov.
\newblock Surgery formula for the renormalized {E}uler characteristic of
  {H}eegaard {F}loer homology.
\newblock arXiv:math.GT/0409294, 2004.

\bibitem[Tan09]{Tange}
Motoo Tange.
\newblock Ozsv\'ath {S}zab\'o's correction term and lens surgery.
\newblock {\em Math. Proc. Cambridge Philos. Soc.}, 146(1):119--134, 2009.

\bibitem[Ter14]{Teragaito-Q}
Masakazu Teragaito.
\newblock Quasi-alternating links and {$Q$}-polynomials.
\newblock {\em J. Knot Theory Ramifications}, 23(12):1450068, 6, 2014.

\bibitem[Ter15]{Teragaito-K}
Masakazu Teragaito.
\newblock Quasi-alternating links and {K}auffman polynomials.
\newblock {\em J. Knot Theory Ramifications}, 24(7):1550038 (17 pages), 2015.

\bibitem[Wal92]{Walker}
Kevin Walker.
\newblock {\em An extension of {C}asson's invariant}, volume 126 of {\em Annals
  of Mathematics Studies}.
\newblock Princeton University Press, Princeton, NJ, 1992.

\end{thebibliography}

\end{document}